\documentclass[reqno]{amsart}
\usepackage{amscd}
\usepackage{amssymb}
\usepackage{amsmath}
\usepackage{amsthm}
\usepackage{latexsym}
\usepackage{upref}
\usepackage{epsfig}
\usepackage{mathrsfs}
\usepackage{float}
\usepackage[colorlinks,urlcolor=black]{hyperref}
\usepackage{graphics,graphicx,color}
\usepackage{floatflt}
\usepackage{indentfirst}
\usepackage[cmtip,arrow]{xy}
\usepackage{pb-diagram,pb-xy}
\usepackage{cite}

\usepackage{enumitem}
\setlist{font=\normalfont}

\newtheorem{theorem}{Theorem}[section]
\newtheorem{proposition}[theorem]{Proposition}
\newtheorem{corollary}[theorem]{Corollary}
\newtheorem{lemma}[theorem]{Lemma}
\theoremstyle{definition}

\newtheorem{definition}[theorem]{Definition}
\newtheorem{example}[theorem]{Example}
\numberwithin{equation}{section}

\newcommand{\abs}[1]{|#1|}
\newcommand{\set}[1]{\{#1\}}
\newcommand{\cset}[2]{\set{{#1}\colon{#2}}}
\renewcommand{\vec}[1]{\mathbf{#1}}
\newcommand{\norm}[1]{\|#1\|}
\newcommand{\gyr}[2]{{\mathrm{gyr}[{#1}]}{#2}}
\renewcommand{\ker}[1]{{\mathrm{ker\,}{#1}}}
\newcommand{\Gyr}[2]{{\mathrm{Gyr}[{#1}]}{#2}}
\newcommand{\igyr}[2]{{\mathrm{gyr^{-1}}[{#1}]}{#2}}
\newcommand{\gen}[1]{\langle#1\rangle}
\newcommand{\Aut}[1]{\mathrm{Aut}\,{(#1)}}
\newcommand{\lbar}{\overline}
\newcommand{\iplus}{\oplus_{\lbar{G}}}
\newcommand{\St}[1]{\mathrm{Stab}\,({#1})}
\newcommand{\sym}[1]{\mathrm{Sym}\,(#1)}
\newcommand{\Id}[1]{\mathrm{Id}\,{(#1)}}
\newcommand{\lcup}[2]{\displaystyle\bigcup_{#1}^{#2}}

\newcommand{\lsum}[2]{\displaystyle\sum_{#1}^{#2}}
\newcommand{\Bp}[1]{\left(#1\right)}
\newcommand{\res}[2]{{#1}\big|_{{#2}}}
\newcommand{\cols}[1]{\mathbf{\mathcal{#1}}}

\newcommand{\R}{\mathbb{R}}
\newcommand{\C}{\mathbb{C}}

\newcommand{\B}{\mathbb{B}}
\newcommand{\D}{\mathbb{D}}

\newcommand{\lamb}{\lambda}

\newcommand{\del}{\delta}

\newcommand{\gam}{\gamma}

\newcommand{\sig}{\sigma}
\newcommand{\vphi}{\varphi}

\begin{document}
\title[Isomorphism Theorems for Gyrogroups and L-Subgyrogroups]{Isomorphism Theorems for Gyrogroups\\ and L-Subgyrogroups}

\author{Teerapong Suksumran}
\address{Department of Mathematics and Computer Science, Faculty of Science,
Chulalongkorn University, Phyathai Road, Patumwan, Bangkok 10330,
Thailand} \email{sk\_teer@yahoo.com \textrm{and
}kwiboonton@gmail.com}

\author{Keng Wiboonton}
\address{}
\email{}

\begin{abstract}
We extend well-known results in group theory to gyrogroups,
\mbox{especially} the isomorphism theorems. We prove that an
arbitrary gyrogroup $G$ induces the gyrogroup structure on the
symmetric group of $G$ so that Cayley's Theorem is obtained.
Introducing the notion of L-subgyrogroups, we show that an
L-subgyrogroup partitions $G$ into left cosets. Consequently, if $H$
is an L-subgyrogroup of a finite gyrogroup $G$, then the order of
$H$ divides the order of $G$.
\end{abstract}
\maketitle

\section{Introduction}
\par Let $c$ be a positive constant representing the speed of light in
vacuum and let $\R^3_c$ denote the $c$-ball of relativistically
admissible velocities, $\R^3_c = \cset{\vec{v}\in
\R^3}{\norm{\vec{v}}<c}$. In \cite{AU2007SD}, Einstein velocity
addition $\oplus_E$ in the $c$-ball is given by the equation
$$\vec{u}\oplus_E\vec{v} =
\frac{1}{1+\frac{\gen{\vec{u},\vec{v}}}{c^2}}\left\{\vec{u}+
\frac{1}{\gam_\vec{u}}\vec{v} +
\frac{1}{c^2}\frac{\gam_\vec{u}}{1+\gam_\vec{u}}\gen{\vec{u},\vec{v}}\vec{u}\right\},
$$
where $\gam_\vec{u}$ is the Lorentz factor given by $\gam_{\vec{u}}
= \dfrac{1}{\sqrt{1-\frac{\norm{\vec{u}}^2}{c^2}}}$.

\par The system $(\R^3_c, \oplus_E)$ does not form a group since
$\oplus_E$ is neither associative nor commutative. Nevertheless,
Ungar showed that $(\R^3_c, \oplus_E)$ is rich in structure and
encodes a group-like structure, namely the gyrogroup structure. He
introduced space rotations $\gyr{\vec{u},\vec{v}}{}$, called
\textit{gyroautomorphisms}, to repair the breakdown of associativity
in $(\R_c^3,\oplus_E)$:
\begin{eqnarray}
\vec{u}\oplus_E(\vec{v}\oplus_E \vec{w}) &=& (\vec{u}\oplus_E \vec{v})\oplus_E\gyr{\vec{u},\vec{v}}{\vec{w}}\nonumber\\
(\vec{u}\oplus_E\vec{v})\oplus_E \vec{w} &=&
\vec{u}\oplus_E(\vec{v}\oplus_E\gyr{\vec{v},\vec{u}}{\vec{w}})\nonumber
\end{eqnarray}
for all $\vec{u}, \vec{v}, \vec{w}\in\R^3_c$. The resulting system
forms a gyrocommutative gyrogroup, called the \textit{Einstein
gyrogroup}, which has been intensively studied in \cite{AKAU2004,
AU2007SD, AU2005, AU2008, AU2009MC, NSAU2013, AU2013, MF2014}.

\par There are close connections between the Einstein gyrogroup
and the Lorentz transformations, as described in \cite[Chapter
11]{AU2008} and \cite{AU2005JGSP}. A Lorentz transformation without
rotation is called a \textit{Lorentz boost}. Let $L(\vec{u})$ and
$L(\vec{v})$ denote Lorentz boosts \mbox{parameterized} by $\vec{u}$
and $\vec{v}$ in $\R^3_c$. The composite of two Lorentz boosts is
not a pure Lorentz boost, but a Lorentz boost followed by a space
rotation:
\begin{equation}\label{eqn: the composite of Lorentz boosts}
L(\vec{u})\circ L(\vec{v}) = L(\vec{u}\oplus_E\vec{v})\circ
\Gyr{\vec{u},\vec{v}}{},
\end{equation}
where $\Gyr{\vec{u},\vec{v}}{}$ is a rotation of spacetime
coordinates induced by the Einstein \mbox{gyroautomorphism}
$\gyr{\vec{u}, \vec{v}}{}$. In this paper, we present an abstract
version of the composition law (\ref{eqn: the composite of Lorentz
boosts}) of Lorentz boosts.

\par Another example of a gyrogroup is the \textit{M\"{o}bius
gyrogroup}, which consists of the complex unit disk $\D =
\cset{z\in\C}{\abs{z}<1}$ with M\"{o}bius addition
\begin{equation}\label{eqn: complex Mobius addition}
a\oplus_M b = \frac{a+b}{1+\bar{a}b}
\end{equation}
for $a, b\in\D$. The M\"{o}bius gyroautomorphisms are given by
\begin{equation}\label{eqn: complex Mobius gyration}
\gyr{a, b}{z} = \frac{1+a\bar{b}}{1+\bar{a}b}z,\hskip0.5cm z\in \D.
\end{equation}
\par Let $\B$ denote the open unit ball of $n$-dimensional Euclidean
space $\R^n$ (or more generally of a real inner product space). In
\cite{AU2008MAA}, Ungar extended M\"{o}bius addition from the
complex unit disk to the unit ball:
\begin{equation}\label{Eqn: Mobius addition on V}\vec{u}\oplus_M\vec{v}
= \dfrac{(1 + 2\gen{\vec{u},\vec{v}} + \norm{\vec{v}}^2)\vec{u} + (1
- \norm{\vec{u}}^2)\vec{v}}{1 + 2\gen{\vec{u},\vec{v}} +
\norm{\vec{u}}^2\norm{\vec{v}}^2}
\end{equation}
for $\vec{u},\vec{v}\in\B$. The unit ball together with M\"{o}bius
addition forms a gyrocommutative gyrogroup, which has been
intensively studied in \cite{MFGR2011, MF2009, JL2010, AU2008MAA,
AU2008, AU2009MC, SKJL2013}.

\par The factorization of M\"{o}bius gyrogroups was comprehensively
studied by Ferreira and Ren in \cite{MF2009, MFGR2011}, in which
they showed that any M\"{o}bius subgyrogroup partitions the
M\"{o}bius gyrogroup into left cosets. The fact that any
subgyrogroup of an arbitrary gyrogroup partitions the gyrogroup is
not stated in the literature, and this is indeed the case, as shown
in Theorem \ref{thm: Sim_H is an equivalence relation}. This result
leads to the introduction of \textit{L-subgyrogroups}. We prove that
an L-subgyrogroup partitions the gyrogroup into left cosets and
consequently obtain a portion of \textit{Lagrange's Theorem}: if $H$
is an L-subgyrogroup of a finite gyrogroup $G$, then the order of
$H$ divides the order of $G$. We also prove the isomorphism theorems
for gyrogroups, in full analogy with their group counterparts.

\section{Basic Properties of Gyrogroups}\label{Sec: Preliminaries}
\par A pair $(G,\oplus)$ consisting of a nonempty set $G$ and a binary
operation $\oplus$ on $G$ is called a \textit{magma}. Let $(G,
\oplus)$ be a magma. A bijection from $G$ to itself is called an
\textit{automorphism} of $G$ if $\vphi(a\oplus b) =
\vphi(a)\oplus\vphi(b)$ for all $a, b\in G$. The set of all
automorphisms of $G$ is denoted by $\Aut{G, \oplus}$. Ungar
formulated the formal definition of a gyrogroup as follows.
\begin{definition}[{\hskip-0.5pt}\cite{AU2008}]\label{Def: Main definition of gyrogroup}
A magma $(G,\oplus)$ is a \textit{gyrogroup} if its binary operation
satisfies the following axioms:
\begin{enumerate}
    \item [(G1)] $\exists0\in G\forall a\in G$, $0\oplus a = a$
    \item [(G2)] $\forall a\in G\exists b\in G$, $b\oplus a = 0$
    \item [(G3)] $\forall a, b\in G\exists\gyr{a,
    b}{}\in\Aut{G,\oplus}\forall c\in G$,
    $$a\oplus (b\oplus c) = (a\oplus b)\oplus\gyr{a, b}{c}$$
    \item [(G4)] $\forall a, b\in G$, $\gyr{a, b}{} = \gyr{a\oplus b,
    b}{}$.
\end{enumerate}
\end{definition}

\newpage

\par The axioms in Definition \ref{Def: Main definition of gyrogroup}
imply the right counterparts.
\begin{theorem}[{\hskip-0.5pt}\cite{AU2008}]
\label{Thm: Alternative def. of gyrogroup} A magma $(G,\oplus)$
forms a gyrogroup if and only if it satisfies the following
properties:
\begin{enumerate}
\item [(g1)] $\exists 0\in G\forall a\in G, 0\oplus a = a$ and $a\oplus 0 =
a$ \hfill\normalfont{(two-sided identity)}
\item [(g2)] $\forall a\in G\exists b\in G, b\oplus a = 0$ \textit{and} $a\oplus b = 0$\hfill\normalfont{(two-sided inverse)}\\
 \textit{For} $a, b, c\in G$, \textit{define}
\begin{equation}\tag{gyrator identity}
 \gyr{a, b}{c} = \ominus(a\oplus
b)\oplus(a\oplus (b\oplus c)),\end{equation} \textit{then}
\item [(g3)] $\gyr{a,
b}{}\in\Aut{G,\oplus}$\hfill\normalfont{(gyroautomorphism)}
\item [(g3a)] $a\oplus (b\oplus c) = (a\oplus b)\oplus\gyr{a, b}{c}$\hfill\normalfont{(left gyroassociative
law)}
\item [(g3b)] $(a\oplus b)\oplus c = a\oplus (b\oplus\gyr{b, a}{c})$\hfill\normalfont{(right gyroassociative
law)}
\item [(g4a)] $\gyr{a, b}{} = \gyr{a\oplus b, b}{}$\hfill\normalfont{(left loop
property)}
\item [(g4b)] $\gyr{a, b}{} = \gyr{a, b\oplus a}{}$.\hfill\normalfont{(right loop
property)}
\end{enumerate}
\end{theorem}
\par The map $\gyr{a, b}{}$ is called the \textit{gyroautomorphism
generated by $a$ and $b$}. By Theorem \ref{Thm: Alternative def. of
gyrogroup}, any gyroautomorphism is completely determined by its
generators via the \textit{gyrator identity}. A gyrogroup $G$ having
the additional property that
\begin{equation}\tag{gyrocommutative law}
a\oplus b = \gyr{a, b}({b\oplus a})
\end{equation} for all $a, b\in G$ is called a \textit{gyrocommutative
gyrogroup}.

\par Many of group theoretic theorems are generalized to the
gyrogroup case with the aid of gyroautomorphisms, see \cite{AU2005,
AU2008} for more details. Some theorems are listed here for easy
reference. To shorten notation, we write $a\ominus b$ instead of
$a\oplus (\ominus b)$.

\begin{theorem}[Theorem 2.11, \cite{AU2005}]\label{Thm: Similar group property -a+b + -b+c = -a + c}
Let $G$ be a gyrogroup. Then
\begin{equation}
(\ominus a \oplus b) \oplus \gyr{\ominus a, b}{(\ominus b\oplus c)}
= \ominus a \oplus c
\end{equation}
for all $a, b, c\in G$.
\end{theorem}

\begin{theorem}[Theorem 2.25, \cite{AU2005}]\label{Thm: Gyrosum inversion law}
For any two elements $a$ and $b$ of a gyrogroup,
\begin{equation}
\ominus (a\oplus b) = \gyr{a, b}{(\ominus b\ominus a)}.
\end{equation}
\end{theorem}

\begin{theorem}[Theorem 2.27, \cite{AU2005}]\label{Thm: Gyroautomorphism are even and inversive symmetric}
The gyroautomorphisms of any gyrogroup $G$ are even,
\begin{equation}
\gyr{\ominus a, \ominus b}{} = \gyr{a, b}{}
\end{equation}
and inversive symmetric,
\begin{equation}
\igyr{a, b}{} = \gyr{b, a}{}
\end{equation}
for all $a, b\in G$.
\end{theorem}

\par Using Theorem \ref{Thm: Gyroautomorphism are even and inversive
symmetric}, one can prove the following proposition.
\begin{proposition}\label{prop: gyr(X) subset X all a, b implies equality}
Let $G$ be a gyrogroup and let $X\subseteq G$. Then the following
are equivalent:
\begin{enumerate}
    \item\label{item: invarint set under gyromap, inclusion} {$\gyr{a, b}{(X)}\subseteq X$ for all $a, b\in G$;}
    \item\label{item: invarint set under gyromap, equality} {$\gyr{a, b}{(X)} = X$ for all $a, b\in G$.}
\end{enumerate}
\end{proposition}

\par The \textit{gyrogroup cooperation} $\boxplus$ is defined by the
equation
\begin{equation}\label{Eqn: Cooperation}
a\boxplus b = a\oplus\gyr{a, \ominus b}{b}, \hskip0.5cm a, b\in G.
\end{equation}
Like groups, every linear equation in a gyrogroup $G$ has a unique
solution in $G$.
\begin{theorem}[Theorem 2.15, \cite{AU2005}]\label{Thm: Linear equations in gyrogroup}
Let $G$ be a gyrogroup and let $a, b\in G$. The unique solution of
the equation $a\oplus x = b$ in $G$ for the unknown $x$ is $x =
\ominus a\oplus b$, and the unique solution of the equation $x\oplus
a = b$ in $G$ for the unknown $x$ is $x = b\boxplus (\ominus a)$.
\end{theorem}

\par The following cancellation laws in gyrogroups are derived as a consequence of
Theorem \ref{Thm: Linear equations in gyrogroup}.
\begin{theorem}[{\hskip-0.5pt}\cite{AU2005}]\label{thm: cancellation law in
gyrogroup} Let $G$ be a gyrogroup. For all $a, b, c\in G$,
\begin{enumerate}
    \item $a\oplus b = a\oplus c$ implies $b = c$\hfill\rm{(general left cancellation
    law)}
    \item $\ominus a\oplus(a\oplus b) = b$\hfill\rm{(left cancellation law)}
    \item $(b\ominus a)\boxplus a = b$\hfill\rm{(right cancellation law I)}
    \item $(b\boxplus (\ominus a))\oplus a = b$. \hfill\rm{(right cancellation law II)}
\end{enumerate}
\end{theorem}

\par It is known in the literature that every gyrogroup forms a left Bol
loop with the $\mathrm{A}_\ell$-property, where the
gyroautomorphisms correspond to \textit{left inner mappings} or
\textit{precession maps}. In fact, gyrogroups and left Bol loops
with the $\mathrm{A}_\ell$-property are equivalent, see for instance
\cite{LS1998AM}.

\par To prove an analog of Cayley's theorem for gyrogroups, we will
make use of the following theorem:
\begin{theorem}[Theorem 1, \cite{MF2011}]\label{Thm: Imposing the gyrogroup structure on a set}
Let $G$ be a gyrogroup, let $X$ be an arbitrary set, and let
$\phi\colon X\to G$ be a bijection. Then $X$ endowed with the
induced operation $a\oplus_X b:= \phi^{-1}(\phi(a)\oplus\phi(b))$
for $a, b\in X$ becomes a gyrogroup.
\end{theorem}

\section{Cayley's Theorem}\label{Sec: Cayley's Theorem}
\par Recall that for $a\in\D$, the map $\tau_a$ that sends a complex number
$z$ to $a\oplus_M z$ defines a M\"{o}bius transformation or
conformal mapping on $\D$, known as a \textit{M\"{o}bius
translation}. In the literature, the following composition law of
M\"{o}bius translations is known:
\begin{equation}\label{eqn: composition law of Mobius translation}
\tau_a\circ\tau_b = \tau_{a\oplus_M b}\circ \gyr{a, b}{}
\end{equation}
for all $a, b\in\D$. In this section, we extend the composition law
(\ref{eqn: composition law of Mobius translation}) to an
\mbox{arbitrary} gyrogroup $G$. We also show that the symmetric
group of $G$ admits the gyrogroup structure induced by $G$, thus
obtaining an analog of Cayley's theorem for gyro-groups.
\par Throughout this section, $G$ and $H$ are arbitrary
gyrogroups.
\par For each $a\in G$, the \textit{left
gyrotranslation by $a$} and the \textit{right gyrotranslation by
$a$} are defined on $G$ by
\begin{equation}\label{Eqn: left and right trans}
L_a\colon x\mapsto a\oplus x\hskip0.3cm\textrm{ and }\hskip0.3cm
R_a\colon x\mapsto x\oplus a.
\end{equation}

\begin{theorem}\label{thm: induced gyrogroup}
Let $G$ be a gyrogroup.
\begin{enumerate}
    \item\label{item: left translation is a permutation} {The left gyrotranslations are permutations of $G$.}
    \item\label{item: Gbar} {Denote the set of all left gyrotranslations of $G$ by $\lbar{G}$.
    The map $\psi\colon G\to \lbar{G}$
    defined by $\psi(a) = L_a$ is bijective. The inverse map
    $\phi:= \psi^{-1}$ fulfills the condition in Theorem \ref{Thm: Imposing the gyrogroup structure on a set}.
    In this case, the induced operation $\oplus_{\lbar{G}}$ is given by
    $$L_a\oplus_{\lbar{G}} L_b = L_{a\oplus b}$$ for all $a, b\in G$.}
    \item\label{item: the composition law} {For all $a, b, c\in G$,
    \begin{equation}\label{eqn: the composition law in Sym (G)}
    L_a\circ L_b = L_{a\oplus b}\circ \gyr{a, b}{}
    \end{equation}
    and
    \begin{equation}\label{eqn: gyrators in induced gyrogroup}
    \mathrm{gyr}_{\lbar{G}}[L_a, L_b]{L_c} = L_{\gyr{a, b}{c}}.
    \end{equation}}
\end{enumerate}
\end{theorem}
\begin{proof} Let $a, b\in G$.
\par (\ref{item: left translation is a permutation}) That $L_a$ is
injective follows from the general left cancellation law. That $L_a$
is surjective follows from Theorem \ref{Thm: Linear equations in
gyrogroup}.
\par (\ref{item: Gbar}) That $\psi$ is bijective is clear. By
Theorem \ref{Thm: Imposing the gyrogroup structure on a set}, the
induced operation is given by
    $$
    L_a\oplus_{\lbar{G}} L_b =
    \psi(\psi^{-1}(L_a)\oplus\psi^{-1}(L_b))=\psi(a\oplus b)= L_{a\oplus b}.
    $$
\par (\ref{item: the composition law}) By the left cancellation law, $L^{-1}_a = L_{\ominus a}$.
    By the gyrator identity, $\gyr{a, b}{} = L_{\ominus(a\oplus b)}\circ L_a\circ L_b$ and
    hence $\gyr{a, b}{} = L^{-1}_{a\oplus b}\circ L_a\circ
    L_b$. It follows that $L_a\circ L_b = L_{a\oplus b}\circ \gyr{a,
    b}{}$. Equation (\ref{eqn: gyrators in induced gyrogroup})
    follows from the gyrator identity.
\end{proof}

\par Let $\St{0}$ denote the set of permutations of $G$
leaving the gyrogroup identity fixed,
$$\St{0} = \cset{\rho\in \sym{G}}{\rho(0) = 0}.$$
It is clear that $\St{0}$ is a subgroup of the symmetric group,
$\sym{G}$, and we have the following inclusions:
$$\cset{\gyr{a, b}{}}{a, b\in G}\subseteq \Aut{G} \leqslant \St{0}\leqslant \sym{G}.$$

\par The next theorem enables us to introduce a binary
operation $\oplus$ on the symmetric group of $G$ so that $\sym{G}$
equipped with $\oplus$ becomes a gyrogroup containing an isomorphic
copy of $G$.

\begin{theorem}\label{thm: induced gyrogroup is a transversal}
For each $\sig\in \sym{G}$, $\sig$ can be written uniquely as $\sig
= L_a\circ \rho$, where $a\in G$ and $\rho\in \St{0}$.
\end{theorem}
\begin{proof} Suppose that $L_a\circ\rho = L_b\circ\eta$, where
$a, b\in G$ and $\rho, \eta\in\St{0}$. Then $a = (L_a\circ\rho)(0) =
(L_b\circ \eta)(0) = b$, which implies $L_a = L_b$ and so $\rho =
\eta$. This proves the uniqueness of factorization. Let $\sig$ be an
arbitrary permutation of $G$. Choose $a = \sig(0)$ and set $\rho =
L_{\ominus a}\circ \sig$. Note that $\rho(0) = L_{\ominus a}(a) =
\ominus a\oplus a = 0$. Hence, $\rho\in\St{0}$. Since $L_{\ominus a}
= L^{-1}_a$, $\sig = L_a\circ\rho$. This proves the existence of
factorization.
\end{proof}

\par The following \textit{commutation relation} determines how to commute
a left gyrotranslation and an automorphism of $G$:
\begin{equation}\label{eqn: commuting relation}
\rho\circ L_a = L_{\rho(a)}\circ \rho
\end{equation}
whenever $\rho$ is an automorphism of $G$.

\par Let $\sig$ and $\tau$ be permutations of $G$. By
Theorem \ref{thm: induced gyrogroup is a transversal}, $\sig$ and
$\tau$ have factorizations $\sig = L_a\circ\gam$ and $\tau =
L_b\circ \del$, where $a, b\in G$ and $\gam,\del\in\St{0}$. Define
an operation $\oplus$ on $\sym{G}$ by
\begin{equation}\label{eqn: gyrogroup operation on Sym (G)}
\sig\oplus\tau = L_{a\oplus b}\circ(\gam\circ\del).
\end{equation}
Because of the uniqueness of factorization, $\oplus$ is a binary
operation on $\sym{G}$. In fact, $(\sym{G}, \oplus)$ forms a
gyrogroup.

\begin{theorem}\label{thm: Sym G with + is a gyrogroup}
$\sym{G}$ is a gyrogroup under the operation defined by (\ref{eqn:
gyrogroup operation on Sym (G)}), and $$L_a\oplus L_b = L_a\iplus
L_b = L_{a\oplus b}$$ for all $a, b\in G$. In particular, the map
$a\mapsto L_a$ defines an injective gyrogroup homomorphism from $G$
into $\sym{G}$.
\end{theorem}
\begin{proof}
Suppose that $\sig = L_a\circ\gam$, $\tau = L_b\circ \del$ and $\rho
= L_c\circ \lamb$, where $a, b, c\in G$ and $\gam, \del, \lamb\in
\St{0}$. The identity map $\Id{G}$ acts as a left identity of
$\sym{G}$ and $L_{\ominus a}\circ\gam^{-1}$ is a left inverse of
$\sig$ with respect to $\oplus$. The gyroautomorphisms of $\sym{G}$
are given by
$$
\gyr{\sig,\tau}{\rho} = (\gyr{L_a, L_b}{L_c})\circ \lamb =
L_{\gyr{a, b}{c}}\circ\lamb.
$$
Since $G$ satisfies the left gyroassociative law and the left loop
property, so does $\sym{G}$.
\end{proof}

\par By Theorem \ref{thm: Sym G with + is a gyrogroup}, the
following version of Cayley's theorem for gyrogroups is
\mbox{immediate}.
\begin{corollary}[Cayley's Theorem]
Every gyrogroup is isomorphic to a subgyro-group of the gyrogroup of
permutations.
\end{corollary}
\begin{proof}
The map $a\mapsto L_a$ defines a gyrogroup isomorphism from $G$ onto
$\lbar{G}$ and $\lbar{G}$ is a subgyrogroup of $\sym{G}$.
\end{proof}

\section{L-Subgyrogroups}\label{Sec: L-subgyrogroup}
\par Throughout this section, $G$ is an arbitrary
gyrogroup.
\par A nonempty subset $H$ of $G$ is a \textit{subgyrogroup}
if $H$ forms a gyrogroup under the operation inherited from $G$ and
the restriction of $\gyr{a, b}{}$ to $H$ is an automorphism of $H$
for all $a, b\in H$. If $H$ is a subgyrogroup of $G$, then we write
$H\leqslant G$ as in the group case.

\begin{proposition}[The Subgyrogroup Criterion]\label{prop: criterion for subgyrogroup}
A nonempty subset $H$ of $G$ is a subgyrogroup if and only if
$\ominus a \in H$ and $a\oplus b\in H$ for all $a, b\in H$.
\end{proposition}
\begin{proof}
Axioms (G1), (G2), (G4) hold trivially. Let $a, b\in H$. By the
gyrator identity, $\gyr{a, b}{(H)}\subseteq H$. Since the
gyroautomorphisms are inversive symmetric (Theorem \ref{Thm:
Gyroautomorphism are even and inversive symmetric}), we also have
the reverse inclusion. Thus, the restriction of $\gyr{a, b}{}$ to
$H$ is an automorphism of $H$ and so axiom (G3) holds.
\end{proof}

\par Let $H$ be a subgyrogroup of $G$. In contrast to groups, the relation
\begin{equation}\label{relation: -a+b}
a\sim b\hskip0.5cm\textrm{if and only if}\hskip0.5cm \ominus a\oplus
b\in H
\end{equation}
does not, in general, define an equivalence relation on $G$.
Nevertheless, we can modify (\ref{relation: -a+b}) to obtain an
equivalence relation on $G$. From this point of view, any
subgyrogroup of $G$ partitions $G$. This leads to the introduction
of L-subgyrogroups.

\par Let $H$ be a subgyrogroup of $G$.
Define a relation $\sim_H$ on $G$ by letting
\begin{equation}\label{relation sim_H}
a \sim_H b\hskip0.3cm\textrm{if and only if}\hskip0.3cm \ominus a
\oplus b\in H \textrm{ and } \gyr{\ominus a, b}{(H)} = H.
\end{equation}

\newpage

\begin{theorem}\label{thm: Sim_H is an equivalence relation}
The relation $\sim_H$ defined by (\ref{relation sim_H}) is an
equivalence relation on $G$.
\end{theorem}
\begin{proof} Let $a, b, c\in G$.
Since $\ominus a \oplus a = 0\in H$ and $\gyr{\ominus a, a}{} =
\Id{G}$, $a\sim_H a$. Hence, $\sim_H$ is reflexive. Suppose that
$a\sim_H b$. By Theorem \ref{Thm: Gyrosum inversion law},
$\gyr{\ominus a, b}{(\ominus b \oplus a)} = \ominus(\ominus a\oplus
b)$. Hence, $\ominus b\oplus a = \igyr{\ominus a,
b}{(\ominus(\ominus a\oplus b))}$, which implies \mbox{$\ominus b
\oplus a\in H$} since $\igyr{\ominus a, b}{(H)} = H$. By Theorem
\ref{Thm: Gyroautomorphism are even and inversive symmetric},
$$\gyr{\ominus a, b}{} = \gyr{\ominus a, \ominus(\ominus b)}{} =
\gyr{a, \ominus b}{} = \igyr{\ominus b, a}{}.$$ Hence, $\gyr{\ominus
b, a}{} = \igyr{\ominus a, b}{}$. Since $\gyr{\ominus a, b}{(H)} =
H$, $\gyr{\ominus b, a}{(H)} = H$ as well. This proves $b\sim_H a$
and so $\sim_H$ is symmetric. Suppose that $a\sim_H b$ and $b\sim_H
c$. By Theorem \ref{Thm: Similar group property -a+b + -b+c = -a +
c}, $\ominus a \oplus c = (\ominus a \oplus b) \oplus \gyr{\ominus
a, b}{(\ominus b\oplus c)}$ and so $\ominus a \oplus c\in H$. Using
the composition law (\ref{eqn: the composition law in Sym (G)}) and
the commutation relation (\ref{eqn: commuting relation}), we have
$\gyr{\ominus a, c}{} = \gyr{\ominus a\oplus b, \gyr{\ominus a,
b}{(\ominus b\oplus c)}}{}\circ\gyr{\ominus a, b}{}\circ\gyr{\ominus
b, c}{}$. This implies $\gyr{\ominus a, c}{(H)} = H$ and so $a\sim_H
c$. This proves $\sim_H$ is transitive.
\end{proof}

\par Let $a\in G$. Let $[a]$ denote the equivalence class of $a$
determined by $\sim_H$. Theorem \ref{thm: Sim_H is an equivalence
relation} says that $\cset{[a]}{a\in G}$ is a partition of $G$. Set
$a\oplus H := \cset{a\oplus h}{h\in H}$, called the \textit{left
coset of $H$ induced by $a$}.

\begin{proposition}\label{prop: equivalence class in left coset}
For each $a\in G$, $[a]\subseteq a\oplus H$.
\end{proposition}
\begin{proof}
If $x\in[a]$, by (\ref{relation sim_H}), $\ominus a\oplus x\in H$.
Hence, $x = a\oplus (\ominus a \oplus x) \in a\oplus H$.
\end{proof}

\par Proposition \ref{prop: equivalence class in left coset} leads to
the notion of L-subgyrogroups:
\begin{definition}\label{def: L-subgyrogroup}
A subgyrogroup $H$ of $G$ is said to be an L-subgyrogroup, denoted
by $H\leqslant_L G$, if $\gyr{a, h}{(H)} = H$ for all $a\in G$ and
$h\in H$.
\end{definition}

\begin{example}\label{Exa: K_16}
\par In \cite[p. 41]{AU2002}, Ungar exhibited the gyrogroup
$K_{16}$ whose addition table is presented in Table \ref{Tab:
Gyrogroup K16}. In $K_{16}$, there is only one nonidentity
gyroauto-morphism, denoted by $A$, whose transformation is given in
cyclic notation by
\begin{equation}\label{eqn: gyroautomorphism of K16}
A = (8\,\, 9)(10\,\, 11)(12\,\, 13)(14\,\, 15).
\end{equation}
The gyration table for $K_{16}$ is presented in Table \ref{Tab:
Gyration of K16}. According to (\ref{eqn: gyroautomorphism of K16}),
$H_1 = \set{0, 1}$, $H_2 = \set{0,1,2,3}$, and $H_3 =
\set{0,1,\dots, 7}$ are easily seen to be L-subgyrogroups of
$K_{16}$. In contrast, $H_4 = \set{0, 8}$ is a non-L-subgyrogroup of
$K_{16}$ since $\gyr{4, 8}{(H_4)}\ne H_4$.
\end{example}

\begin{table}[ht]
\centering
\begin{tabular}{|c|cccccccccccccccc|}
\hline $\oplus$ & 0 & 1 & 2 & 3 & 4 & 5 & 6 & 7 & 8 & 9 & 10 & 11 &
12 & 13 & 14 & 15 \\ \hline 0 & 0 & 1 & 2 & 3 & 4 & 5 & 6 & 7 & 8 &
9 & 10 &
11 & 12 & 13 & 14 & 15 \\
1 & 1 & 0 & 3 & 2 & 5 & 4 & 7 & 6 & 9 & 8 & 11 & 10 &
13 & 12 & 15 & 14 \\
2 & 2 & 3 & 1 & 0 & 6 & 7 & 5 & 4 & 11 & 10 & 8 & 9 &
15 & 14 & 12 & 13 \\
3 & 3 & 2 & 0 & 1 & 7 & 6 & 4 & 5 & 10 & 11 & 9 & 8 &
14 & 15 & 13 & 12 \\
4 & 4 & 5 & 6 & 7 & 3 & 2 & 0 & 1 & 15 & 14 & 12 & 13 &
9 & 8 & 11 & 10 \\
5 & 5 & 4 & 7 & 6 & 2 & 3 & 1 & 0 & 14 & 15 & 13 & 12 &
8 & 9 & 10 & 11 \\
6 & 6 & 7 & 5 & 4 & 0 & 1 & 2 & 3 & 13 & 12 & 15 & 14 &
10 & 11 & 9 & 8 \\
7 & 7 & 6 & 4 & 5 & 1 & 0 & 3 & 2 & 12 & 13 & 14 & 15 &
11 & 10 & 8 & 9 \\
8 & 8 & 9 & 10 & 11 & 12 & 13 & 14 & 15 & 0 & 1 & 2 & 3 &
4 & 5 & 6 & 7 \\
9 & 9 & 8 & 11 & 10 & 13 & 12 & 15 & 14 & 1 & 0 & 3 & 2 &
5 & 4 & 7 & 6 \\
10 & 10 & 11 & 9 & 8 & 14 & 15 & 13 & 12 & 3 & 2 & 0 & 1 &
7 & 6 & 4 & 5 \\
11 & 11 & 10 & 8 & 9 & 15 & 14 & 12 & 13 & 2 & 3 & 1 & 0 &
6 & 7 & 5 & 4 \\
12 & 12 & 13 & 14 & 15 & 11 & 10 & 8 & 9 & 6 & 7 & 5 & 4 &
0 & 1 & 2 & 3 \\
13 & 13 & 12 & 15 & 14 & 10 & 11 & 9 & 8 & 7 & 6 & 4 & 5 &
1 & 0 & 3 & 2 \\
14 & 14 & 15 & 13 & 12 & 8 & 9 & 10 & 11 & 4 & 5 & 6 & 7 &
3 & 2 & 0 & 1 \\
15 & 15 & 14 & 12 & 13 & 9 & 8 & 11 & 10 & 5 & 4 & 7 & 6 & 2 & 3 &
1& 0\\
\hline
\end{tabular}\vskip5pt
\caption{Addition table for the gyrogroup $K_{16}$,
\cite{AU2002}}\label{Tab: Gyrogroup K16}
\end{table}

\begin{table}[ht]
  \centering
  \begin{tabular}{|c|cccccccccccccccc|}
\hline $\textrm{gyr}$ & 0 & 1 & 2 & 3 & 4 & 5 & 6 & 7 & 8 & 9 & 10 &
11 & 12 & 13 & 14 & 15 \\ \hline 0 & $I$ & $I$ & $I$ & $I$ & $I$ &
$I$ & $I$ & $I$ & $I$ &
$I$ & $I$ & $I$ & $I$ & $I$ & $I$ & $I$ \\
1 & $I$ & $I$ & $I$ & $I$ & $I$ & $I$ & $I$ & $I$ & $I$ &
$I$ & $I$ & $I$ & $I$ & $I$ & $I$ & $I$ \\
2 & $I$ & $I$ & $I$ & $I$ & $I$ & $I$ & $I$ & $I$ & $I$ &
$I$ & $I$ & $I$ & $I$ & $I$ & $I$ & $I$ \\
3 & $I$ & $I$ & $I$ & $I$ & $I$ & $I$ & $I$ & $I$ & $I$ &
$I$ & $I$ & $I$ & $I$ & $I$ & $I$ & $I$ \\
4 & $I$ & $I$ & $I$ & $I$ & $I$ & $I$ & $I$ & $I$ & $A$ &
$A$ & $A$ & $A$ & $A$ & $A$ & $A$ & $A$ \\
5 & $I$ & $I$ & $I$ & $I$ & $I$ & $I$ & $I$ & $I$ & $A$ &
$A$ & $A$ & $A$ & $A$ & $A$ & $A$ & $A$ \\
6 & $I$ & $I$ & $I$ & $I$ & $I$ & $I$ & $I$ & $I$ & $A$ &
$A$ & $A$ & $A$ & $A$ & $A$ & $A$ & $A$ \\
7 & $I$ & $I$ & $I$ & $I$ & $I$ & $I$ & $I$ & $I$ & $A$ &
$A$ & $A$ & $A$ & $A$ & $A$ & $A$ & $A$ \\
8 & $I$ & $I$ & $I$ & $I$ & $A$ & $A$ & $A$ & $A$ & $I$ &
$I$ & $I$ & $I$ & $A$ & $A$ & $A$ & $A$ \\
9 & $I$ & $I$ & $I$ & $I$ & $A$ & $A$ & $A$ & $A$ & $I$ &
$I$ & $I$ & $I$ & $A$ & $A$ & $A$ & $A$ \\
10 & $I$ & $I$ & $I$ & $I$ & $A$ & $A$ & $A$ & $A$ & $I$ &
$I$ & $I$ & $I$ & $A$ & $A$ & $A$ & $A$ \\
11 & $I$ & $I$ & $I$ & $I$ & $A$ & $A$ & $A$ & $A$ & $I$ &
$I$ & $I$ & $I$ & $A$ & $A$ & $A$ & $A$ \\
12 & $I$ & $I$ & $I$ & $I$ & $A$ & $A$ & $A$ & $A$ & $A$ &
$A$ & $A$ & $A$ & $I$ & $I$ & $I$ & $I$ \\
13 & $I$ & $I$ & $I$ & $I$ & $A$ & $A$ & $A$ & $A$ & $A$ &
$A$ & $A$ & $A$ & $I$ & $I$ & $I$ & $I$ \\
14 & $I$ & $I$ & $I$ & $I$ & $A$ & $A$ & $A$ & $A$ & $A$ &
$A$ & $A$ & $A$ & $I$ & $I$ & $I$ & $I$ \\
15 & $I$ & $I$ & $I$ & $I$ & $A$ & $A$ & $A$ & $A$ & $A$ & $A$ & $A$
& $A$ & $I$ & $I$ & $I$ & $I$\\
\hline
\end{tabular}\vskip5pt
\caption{Gyration table for $K_{16}$. Here, $A$ is given by
(\ref{eqn: gyroautomorphism of K16}) and $I$ stands for the identity
transformation, \cite{AU2002}}\label{Tab: Gyration of K16}
\end{table}

The importance of L-subgyrogroups lies in the following results.
\begin{proposition}\label{Prop: Equivalence class for L-subgyrogroup}
If $H\leqslant_L G$, then $[a] = a\oplus H$ for all $a\in G$.
\end{proposition}
\begin{proof}
Assume that $H\leqslant_L G$. By Proposition \ref{prop: equivalence
class in left coset}, $[a]\subseteq a\oplus H$. If $x = a\oplus h$
for some $h\in H$, then $\ominus a \oplus x = h$ is in $H$. The left
and right loop properties together imply $\gyr{\ominus a, x}{} =
\gyr{h, a}{} = \igyr{a, h}{}$. By assumption, $\gyr{a, h}(H) = H$,
which implies $\gyr{\ominus a, x}{(H)} = \igyr{a, h}{(H)} = H$.
Hence, $a\sim_H x$ and so $x\in [a]$. This establishes the reverse
inclusion.
\end{proof}

\begin{theorem}\label{Thm: Left cosets from a partition of G}
If $H$ is an L-subgyrogroup of a gyrogroup $G$, then the set
$$\cset{a\oplus H}{a\in G}$$ forms a disjoint partition of $G$.
\end{theorem}
\begin{proof}
This follows directly from Theorem \ref{thm: Sim_H is an equivalence
relation} and Proposition \ref{Prop: Equivalence class for
L-subgyrogroup}.
\end{proof}

\par In light of Theorem \ref{Thm: Left cosets from a partition of G}, we
derive the following version of Lagrange's theorem for
L-subgyrogroups.

\begin{theorem}[Lagrange's Theorem for L-Subgyrogroups]\label{thm: LT for L-subgyrogroup}
In a finite gyrogroup $G$, if $H\leqslant_L G$, then $\abs{H}$
divides $\abs{G}$.
\end{theorem}
\begin{proof}
Being a finite gyrogroup, $G$ has a finite number of left cosets,
namely $a_1\oplus H$, $a_2\oplus H$, $\dots$, $a_n\oplus H$. Since
$\abs{a_i\oplus H} =\abs{H}$ for $i = 1, 2, \ldots, n$, it follows
that $\abs{G} = \Big|\lcup{i=1}{n}a_i\oplus H\Big| =
\lsum{i=1}{n}\abs{a_i\oplus H} = n\abs{H}$, which completes the
proof.
\end{proof}

\par Let us denote by $[G\colon H]$ the number of left cosets of $H$ in $G$.

\begin{corollary}\label{cor: |G| = [G:H]|H|}
In a finite gyrogroup $G$, if $H\leqslant_L G$, then $\abs{G} =
[G\colon H]\abs{H}$.
\end{corollary}

\par For a \textit{non}-L-subgyrogroup $K$ of $G$, it is no longer true
that the left cosets of $K$ partition $G$. Moreover, the formula
$\abs{G} = [G\colon K]\abs{K}$ is not true, in general.

\section{Isomorphism Theorems}\label{Sec: Isomorphism}
\par A map $\vphi\colon G\to H$ between gyrogroups is
called a \textit{gyrogroup homomorphism} if $\vphi(a\oplus b) =
\vphi(a)\oplus\vphi(b)$ for all $a, b\in G$. A bijective gyrogroup
homomorphism is called a \textit{gyrogroup isomorphism}. We say that
$G$ and $H$ are \textit{isomorphic gyrogroups}, written $G\cong H$,
if there exists a gyrogroup isomorphism from $G$ to $H$. The next
proposition lists basic properties of gyrogroup homomorphisms.
\begin{proposition}\label{prop: property of homomorphism}
Let $\vphi\colon G\to H$ be a homomorphism of gyrogroups.
\begin{enumerate}
    \item\label{item: preserving identity} {$\vphi(0) = 0$.}
    \item\label{item: preserving inverse} {$\vphi(\ominus a) = \ominus \vphi(a)$ for all $a\in
    G$.}
    \item\label{item: preserving gyration} {$\vphi(\gyr{a, b}{c}) = \gyr{\vphi(a),\vphi(b)}{\vphi(c)}$ for all $a, b, c\in
    G$.}
    \item\label{item: preserving coaddition} {$\vphi(a\boxplus b) = \vphi(a)\boxplus\vphi(b)$ for all $a, b\in
    G$.}
\end{enumerate}
\end{proposition}

\par The proof of the following two propositions is routine,
using the subgyrogroup criterion and the definition of an
L-subgyrogroup.
\begin{proposition}\label{prop: image of subgyrogroup is a subgyrogroup}
Let $\vphi\colon G\to H$ be a gyrogroup homomorphism. If $K\leqslant
G$, then $\vphi(K)\leqslant H$. If $K\leqslant_L G$ and if $\vphi$
is surjective, then $\vphi(K)\leqslant_L H$.
\end{proposition}

\begin{proposition}\label{prop: inverse image is a subgyrogroup}
Let $\vphi\colon G\to H$ be a gyrogroup homomorphism. If $K\leqslant
H$, then $\vphi^{-1}(K)\leqslant G$. If $K\leqslant_L H$, then
$\vphi^{-1}(K)\leqslant_L G$.
\end{proposition}

\par Let $\vphi\colon G\to H$ be a gyrogroup homomorphism.
The kernel of $\vphi$ is defined to be the inverse image of the
trivial subgyrogroup $\set{0}$ under $\vphi$, hence is a
subgyro-group. The kernel of $\vphi$ is invariant under the
gyroautomorphisms of $G$, that is, $\gyr{a,
b}{(\ker{\vphi})}\subseteq \ker{\vphi}$ for all $a, b\in G$. By
Proposition \ref{prop: gyr(X) subset X all a, b implies equality},
$\gyr{a, b}{(\ker{\vphi})} =  \ker{\vphi}$ for all $a, b\in G$ and
so $\ker{\vphi}$ is an L-subgyrogroup of $G$. From this the relation
(\ref{relation sim_H}) becomes
\begin{equation}\label{relation sim_kernel}
a\sim_{\ker{\vphi}} b \hskip0.3cm\textrm{if and only if}\hskip0.3cm
\ominus a\oplus b \in \ker{\vphi}
\end{equation}
for all $a, b\in G$. More precisely, we have the following
proposition.

\begin{proposition}\label{prop: equivalence of sim_kernel}
Let $\vphi\colon G\to H$ be a gyrogroup homomorphism. For all $a,
b\in G$, the following are equivalent:
\begin{enumerate}
    \item\label{item: sim_kernel} {$a\sim_{\ker{\vphi}} b$}
    \item\label{item: -a+b} {$\ominus a\oplus b \in \ker{\vphi}$}
    \item\label{item: phi(a) = phi(b)} {$\vphi(a)=\vphi(b)$}
    \item\label{item: coset equal} {$a\oplus \ker{\vphi} = b\oplus \ker{\vphi}$.}
\end{enumerate}
\end{proposition}

\par In view of Proposition \ref{prop: equivalence of sim_kernel},
we define a binary operation on the set $G/\ker{\vphi}$ of left
cosets of $\ker{\vphi}$ in the following natural way:
\begin{equation}\label{eqn: quotient gyrogroup operation}
(a\oplus\ker{\vphi})\oplus (b\oplus \ker{\vphi}) = (a\oplus b)\oplus
\ker{\vphi}, \hskip0.5cm a, b\in G.
\end{equation}
The resulting system forms a gyrogroup, called a \textit{quotient
gyrogroup}.

\begin{theorem}\label{thm: G/kernel is a gyrogroup}
If $\vphi\colon G\to H$ is a gyrogroup homomorphism, then
$G/\ker{\vphi}$ with operation defined by (\ref{eqn: quotient
gyrogroup operation}) is a gyrogroup.
\end{theorem}
\begin{proof}
Set $K = \ker{\vphi}$. The coset $0\oplus K$ is a left identity of
$G/K$. The coset $(\ominus a)\oplus K$ is a left inverse of $a\oplus
K$. For $X = a\oplus K, Y = b\oplus K\in G/K$, the gyroautomorphism
generated by $X$ and $Y$ is given by
$$\gyr{X, Y}{(c\oplus K)} = (\gyr{a, b}{c})\oplus K$$
for $c\oplus K\in G/K$.
\end{proof}

\par The map $\Pi\colon G\to G/\ker{\vphi}$ given by $\Pi(a) = a\oplus
\ker{\vphi}$ defines a surjective \mbox{gyrogroup} homomorphism,
which will be referred to as the \textit{canonical projection}. In
light of Theorem \ref{thm: G/kernel is a gyrogroup}, the first
isomorphism theorem for gyrogroups follows.

\begin{theorem}[The First Isomorphism Theorem]\label{thm: the 1st isomorphism
theorem} If $\vphi\colon G\to H$ is a gyrogroup \mbox{homomorphism},
then $G/\ker{\vphi}\cong \vphi(G)$ as gyrogroups.
\end{theorem}
\begin{proof}
Set $K = \ker{\vphi}$. Define $\phi\colon G/K\to \vphi(G)$ by
$\phi(a\oplus K) = \vphi(a)$. By Proposition \ref{prop: equivalence
of sim_kernel}, $\phi$ is well defined and injective. A direct
computation shows that $\phi$ is a gyrogroup isomorphism from $G/K$
onto $\vphi(G)$.
\end{proof}

\par It is known that a subgroup of a group is normal
if and only if it is the kernel of some group homomorphism. This
characterization of a normal subgroup allows us to define a normal
subgyrogroup in a similar fashion, as follows. A subgyrogroup $N$ of
a gyrogroup $G$ is \textit{normal in $G$}, denoted by $N\unlhd G$,
if it is the kernel of a gyrogroup homomorphism of $G$.

\begin{lemma}\label{lem: A+B is a subgyrogroup}
Let $G$ be a gyrogroup. If $A\leqslant G$ and $B\unlhd G$, then
$$A\oplus B := \cset{a\oplus b}{a\in A, b\in B}$$
forms a subgyrogroup of $G$.
\end{lemma}
\begin{proof}
By assumption, $B = \ker\phi$, where $\phi$ is a gyrogroup
homomorphism of $G$. Using Theorem \ref{Thm: Linear equations in
gyrogroup}, one can prove that $B\oplus a = a\oplus B$ for all $a\in
G$.
\par Let $x = a\oplus b$, with $a\in A$, $b\in B$.
Since $\phi(\gyr{a, b}{\ominus a}) = \gyr{\phi(a), 0}{\phi(\ominus
a)} = \phi(\ominus a)$, we have $\gyr{a, b}{\ominus a} = \ominus a
\oplus b_1$ for some $b_1\in B$. Set $b_2 = \gyr{a, b}{\ominus b}$.
Since $b_2\in B$ and $B\oplus (\ominus a) = (\ominus a)\oplus B$,
there is an element $b_3\in B$ for which $b_2\ominus a = \ominus
a\oplus b_3$. The left and right loop properties together imply
$\ominus x = \ominus a\oplus \Bp{b_3\oplus \gyr{b_3, \ominus
a}{(\gyr{b_2, \ominus a}{b_1})}}$, whence $\ominus x$ belongs to
$A\oplus B$.
\par For $x, y\in A\oplus B$, we have $x = a\oplus b$ and
$y = c\oplus d$ for some $a, c\in A$, $b, d\in B$. Since
$\phi(b\oplus\gyr{b, a}{(c\oplus d)}) =
\phi(b)\oplus\gyr{\phi(b),\phi(a)}{(\phi(c)\oplus\phi(d))} =
\phi(c)$, we have $b\oplus\gyr{b, a}{(c\oplus d)} = c\oplus b_1$ for
some $b_1\in B$. The left and right loop properties together imply $
x\oplus y = (a\oplus c)\oplus\gyr{a, c}{b_1}$, whence $x\oplus y$
belongs to $A\oplus B$. This proves $A\oplus B\leqslant G$.
\end{proof}

\begin{theorem}[The Second Isomorphism Theorem]\label{thm: the 2nd isomorphism theorem}
Let $G$ be a gyrogroup and let $A, B\leqslant G$. If $B\unlhd G$,
then $A\cap B\unlhd A$ and $(A\oplus B)/B\cong A/(A\cap B)$ as
gyrogroups.
\end{theorem}
\begin{proof}
As in Lemma \ref{lem: A+B is a subgyrogroup}, $B = \ker{\phi}$. Note
that $A\cap B\unlhd A$ since $\ker{\res{\phi}{A}} = A\cap B$. Hence,
$A/(A\cap B)$ admits the quotient gyrogroup structure.
\par Define $\vphi\colon A\oplus B\to A/(A\cap B)$ by
$\vphi(a\oplus b) = a\oplus(A\cap B)$ for $a\in A$ and $b\in B$. To
see that $\vphi$ is well defined, suppose that $a\oplus b =
a_1\oplus b_1$, where $a, a_1\in A$ and $b, b_1\in B$. Note that
$b_1 = \ominus a_1 \oplus(a\oplus b) = (\ominus a_1\oplus
a)\oplus\gyr{\ominus a_1, a}{b}$. Set $b_2 = \ominus\gyr{\ominus
a_1, a}{b}$. Then $b_2\in B$ and $b_1 = (\ominus a_1\oplus a)\ominus
b_2$. The right cancellation law I gives $\ominus a_1 \oplus a =
b_1\boxplus b_2 = b_1\oplus \gyr{b_1, \ominus b_2}{b_2}$, which
implies $\ominus a_1\oplus a\in A\cap B$. By Proposition \ref{prop:
equivalence of sim_kernel}, $a_1\oplus (A\cap B) = a\oplus (A\cap
B)$.
\par As we computed in the lemma,
if $a, c\in A$ and $b, d\in B$, then $$(a\oplus b)\oplus(c\oplus d)
= (a\oplus c)\oplus \gyr{a, c}{\tilde{b}}$$ for some $\tilde{b}\in
B$. Hence, $ \vphi((a\oplus b)\oplus(c\oplus d)) =  (a\oplus
c)\oplus A\cap B = \vphi(a\oplus b)\oplus\vphi(c\oplus d)$. This
proves $\vphi\colon A\oplus B\to A/(A\cap B)$ is a surjective
gyrogroup homomorphism whose kernel is $\cset{a\oplus b}{a\in A,
b\in B, a\in A\cap B} = B$. Thus, $B\unlhd A\oplus B$ and by the
first isomorphism theorem, $(A\oplus B)/B \cong A/(A\cap B)$.
\end{proof}

\begin{theorem}[The Third Isomorphism Theorem]\label{thm: the 3rd isomorphism theorem}
Let $G$ be a gyrogroup and let $H, K$ be normal subgyrogroups of $G$
such that $H\subseteq K$. Then $K/H\unlhd G/H$ and $(G/H)/(K/H)\cong
G/K$ as gyrogroups.
\end{theorem}
\begin{proof}
Let $\phi$ and $\psi$ be gyrogroup homomorphisms of $G$ such that
$\ker{\phi} = H$ and $\ker{\psi} = K$. Define $\vphi\colon G/H\to
G/K$ by $\vphi(a\oplus H) = a\oplus K$ for $a\in G$. Note that
$\vphi$ is well defined since $H\subseteq K$. Furthermore, $\vphi$
is a surjective \mbox{gyrogroup} homomorphism such that $\ker{\vphi}
= K/H$. Hence, $K/H\unlhd G/H$. By the first isomorphism theorem,
$(G/H)/(K/H) \cong G/K$.
\end{proof}

\begin{theorem}[The Lattice Isomorphism Theorem]\label{thm: the 4th isomorphism theorem}
Let $G$ be a gyrogroup and let $N\unlhd G$. There is a bijection
$\Phi$ from the set of subgyrogroups of $G$ containing $N$ onto the
set of subgyrogroups of $G/N$. The bijection $\Phi$ has the
following properties:
\begin{enumerate}
    \item\label{item: preserve inclusion} {$A\subseteq B$ if and only if
    $\Phi(A)\subseteq\Phi(B)$}
    \item\label{item: preserve L-subgyrogroup} {$A\leqslant_L G$ if and only if
    $\Phi(A)\leqslant_L G/N$}
    \item\label{item: preserve normal subgyrogroup} {$A\unlhd G$ if and only if
    $\Phi(A)\unlhd G/N$}
\end{enumerate}
for all subgyrogroups $A$ and $B$ of $G$ containing $N$.
\end{theorem}
\begin{proof}
Set $\cols{S} = \cset{K\subseteq G}{K\leqslant G\textrm{ and
}N\subseteq K}$. Let $\cols{T}$ denote the set of subgyro-groups of
$G/N$. Define a map $\Phi$ by $\Phi(K) = K/N$ for $K\in\cols{S}$. By
Proposition \ref{prop: image of subgyrogroup is a subgyrogroup},
$\Phi(K) = K/N = \Pi(K)$ is a subgyrogroup of $G/N$, where
$\Pi\colon G\to G/N$ is the canonical projection. Hence, $\Phi$ maps
$\cols{S}$ to $\cols{T}$.
\par Assume that $K_1/N = K_2/N$, with $K_1, K_2$ in
$\cols{S}$. For $a\in K_1$, $a\oplus N\in K_2/N$ implies
\mbox{$a\oplus N = b\oplus N$} for some $b\in K_2$. Hence, $\ominus
b\oplus a\in N$. Since $N\subseteq K_2$, $\ominus b\oplus a\in K_2$,
which implies $a = b\oplus (\ominus b\oplus a)\in K_2$. This proves
$K_1\subseteq K_2$. By interchanging the roles of $K_1$ and $K_2$,
one obtains similarly that $K_2\subseteq K_1$. Hence, $K_1 = K_2$
and $\Phi$ is injective.
\par Let $Y$ be an arbitrary subgyrogroup of $G/N$.
By Proposition \ref{prop: inverse image is a subgyrogroup},
$$\Pi^{-1}(Y) = \cset{a\in G}{a\oplus N\in Y}$$ is a subgyrogroup of $G$
containing $N$ for $a\in N$ implies $a\oplus N = 0\oplus N\in Y$.
Because $\Phi(\Pi^{-1}(Y)) = Y$, $\Phi$ is surjective. This proves
$\Phi$ defines a bijection from $\cols{S}$ onto $\cols{T}$.
\par The proof of Item \ref{item: preserve inclusion} is straightforward.
From Propositions \ref{prop: image of subgyrogroup is a
subgyrogroup} and \ref{prop: inverse image is a subgyrogroup}, we
have Item \ref{item: preserve L-subgyrogroup}. To prove Item
\ref{item: preserve normal subgyrogroup}, suppose that $A\unlhd G$.
Then $A = \ker{\psi}$, where $\psi\colon G\to H$ is a gyrogroup
homomorphism. Define $\vphi\colon G/N\to H$ by $\vphi(a\oplus N) =
\psi(a)$. Since $N\subseteq A$, $\vphi$ is well defined. Also,
$\vphi$ is a gyrogroup homomorphism. Since $\ker{\vphi} = A/N$, we
have $A/N\unlhd G/N$. Suppose conversely that $\Phi(A)\unlhd G/N$.
Then $A/N = \ker{\phi}$, where $\phi$ is a gyrogroup homomorphism of
$G/N$. Set $\vphi = \phi\circ\Pi$. Thus, $\vphi$ is a gyrogroup
homomorphism of $G$ with kernel $A$ and hence $A\unlhd G$.
\end{proof}

\section*{Acknowledgements}
We would like to thank Abraham A. Ungar for the suggestion of
gyrogroup $K_{16}$ supporting that L-subgyrogroups do exist and for
his helpful comments. This work was financially supported by
National Science Technology Development Agency (NSTDA), Thailand,
via Junior Science Talent Project (JSTP), under grant No.
JSTP-06-55-32E.

\bibliographystyle{amsplain}
\bibliography{Isomorphism_Theorem}
\end{document}